\documentclass[leqno,12pt]{amsart}
\usepackage{amsthm}
\usepackage{amssymb}
\usepackage{verbatim}
\usepackage{enumerate}
\numberwithin{equation}{section}

\theoremstyle{plain}
\newtheorem{thm}[equation]{Theorem}

\newtheorem{cor}[equation]{Corollary}
\newtheorem{lemma}[equation]{Lemma}

\theoremstyle{definition}

\newtheorem{remark}[equation]{Remark}

\newcommand{\F}{\mathbb F}

\renewcommand{\bar}[1]{#1\llap{$\overline{\phantom{\rm#1}}$}}

\title[Permutation polynomials from R\'edei functions]{Permutation polynomials on $\F_q$ induced from R\'edei function bijections on subgroups of $\F_q^*$}

\author{Michael E. Zieve}
\address{
  Department of Mathematics,
  University of Michigan,
  Ann Arbor, MI 48109--1043,
  USA
}
\email{zieve@umich.edu}
\urladdr{www.math.lsa.umich.edu/$\sim$zieve/}

\date{07 October 2013}

\begin{document}

\begin{abstract}
We construct classes of permutation polynomials over $\F_{Q^2}$ by exhibiting classes of low-degree rational functions over $\F_{Q^2}$ which induce bijections on the set of $(Q+1)$-th roots of unity.  As a consequence, we prove two conjectures about permutation trinomials from a recent paper by Tu, Zeng, Hu and Li.
\end{abstract}

\maketitle

\section{Introduction}

A polynomial $f(x)\in\F_q[x]$ is called a \emph{permutation polynomial} if the function
$\alpha\mapsto f(\alpha)$ induces a permutation of $\F_q$.  Many authors have determined collections of
permutation polynomials having especially simple shapes.  The vast majority of these polynomials have
the form $x^r h(x^d)$ where $h\in\F_q[x]$ and $d>1$ is a divisor of $q-1$.  The reason this form is special is that a general result (see Lemma~\ref{lem}) asserts that
$x^r h(x^d)$ permutes $\F_q$ if and only if $\gcd(r,d)=1$ and $x^r h(x)^d$ permutes the set of $(q-1)/d$-th roots of unity in $\F_q^*$.  This leads to the question of producing collections of polynomials which permute the set $\mu_k$ of $k$-th roots of unity in $\F_q$ for certain values of $k$.  
In this paper we introduce a new variant of this construction, in which the induced function on $\mu_k$ is most naturally given by a rational function rather than a polynomial.  It is perhaps surprising that there are permutations of $\mu_k$ which can be represented by a rational function having an especially simple form, but which cannot be represented by an especially simple polynomial.  We obtain the following classes of permutation polynomials.

\begin{thm} \label{thmb} Let $Q$ be a prime power, let $n>0$ and $k\ge 0$ be integers,
and let $\beta,\gamma\in\F_{Q^2}$ satisfy $\beta^{Q+1}=1$ and $\gamma^{Q+1}\ne 1$.  Then
\[
f(x):=x^{n+k(Q+1)} \cdot\left( (\gamma x^{Q-1}-\beta)^n - \gamma (x^{Q-1}-\gamma^Q\beta)^n \right)
\]
permutes\/ $\F_{Q^2}$ if and only if $\gcd(n+2k,Q-1)=1$ and $\gcd(n,Q+1)=1$.
\end{thm}

\begin{thm} \label{thma} Let $Q$ be a prime power, let $n,k$ be integers with $n>0$ and $k\ge 0$,
and let $\beta,\delta\in\F_{Q^2}$ satisfy $\beta^{Q+1}=1$ and $\delta\notin\F_Q$.  Then
\[
f(x):=x^{n+k(Q+1)} \cdot \left((\delta x^{Q-1}-\beta \delta^Q)^n - \delta(x^{Q-1}-\beta)^n\right)
\]
permutes\/ $\F_{Q^2}$ if and only if $\gcd(n(n+2k),Q-1)=1$.
\end{thm}

The following corollary illustrates these results in the special case $n=3$, for certain values of $\beta,\gamma,\delta$.

\begin{cor} \label{main}
Let $Q$ be a prime power, and let $k$ be a nonnegative integer.  The polynomial
$g(x):=x^{k(Q+1)+3} + 3x^{k(Q+1)+Q+2} - x^{k(Q+1)+3Q}$ permutes\/
$\F_{Q^2}$ if and only if $\gcd(2k+3,Q-1)=1$ and $3\nmid Q$.
\end{cor}

Specializing even further to the values $k=Q-3$, $k=1$, and $k=0$ yields the following consequence.

\begin{cor} \label{maincor}
Let $Q$ be a prime power with $3\nmid Q$.  Then
\begin{enumerate}
\item $x^Q+3x^{2Q-1}-x^{Q^2-Q+1}$ is a permutation polynomial over\/ $\F_{Q^2}$.
\item $x^{Q+4}+3x^{2Q+3}-x^{4Q+1}$ is a permutation polynomial over\/ $\F_{Q^2}$ if and only if
$Q\not\equiv 1\pmod{5}$.
\item $x^3+3x^{Q+2}-x^{3Q}$ is a permutation polynomial over\/ $\F_{Q^2}$ if and only if
$Q\equiv 2\pmod{3}$.
\end{enumerate}
\end{cor}

In case $Q=2^{2m+1}$, the first two parts of this corollary were conjectured by
Tu, Zeng, Hu and Li \cite{TZHL}.  Conversely, these conjectures were the impetus which led to the
present paper.

The proofs of our results rely on exhibiting certain permutations of the set of
$(Q+1)$-th roots of unity in $\F_{Q^2}$.  The permutations we exhibit are represented by
\emph{R\'edei functions}, namely, rational functions over a field which are conjugate to $x^n$
over an extension field.  For further results about such functions, see for instance
\cite{Carlitz,Nobauer,Redei} and \cite[Ch.\ 5]{GMS}.  Although permutation polynomials on
subgroups of $\F_q^*$ have also been studied \cite{Brison}, the present paper is the first to
examine R\'edei permutations of any such subgroup which is not itself the multiplicative group of
a subfield.

We prove Theorems~\ref{thmb} and \ref{thma} in the next two sections, and deduce the corollaries
in Section~4.
We conclude this paper by using our
approach to give a very simple proof of a substantial generalization of the main result of
\cite{TZHL}.

%##################################################################
%##################################################################

\section{Proof of Theorem~\ref{thmb}}

In this section we prove Theorem~\ref{thmb}.  Throughout this section, $Q$ is a prime power, and
we write $\mu_d$ for the set of $d$-th roots of unity in the algebraic closure $\bar\F_Q$ of $\F_Q$.  We begin with a lemma
exhibiting the permutations of $\mu_{Q+1}$ induced by degree-one rational functions.

\begin{lemma} \label{q+1toq+1} Let $Q$ be a prime power, and let $\ell(x)\in\bar\F_Q(x)$ be a degree-one rational function.
Then $\ell$ induces a bijection on $\mu_{Q+1}$ if and only if $\ell(x)$ equals either
\begin{itemize}
\item $\beta/x$ with $\beta\in\mu_{Q+1}$, or
\item $(x-\gamma^Q\beta)/(\gamma x-\beta)$ with $\beta\in\mu_{Q+1}$ and $\gamma\in\F_{Q^2}\setminus\mu_{Q+1}$.
\end{itemize}
\end{lemma}

\begin{proof}
If $\ell(x)$ induces a bijection on $\mu_{Q+1}$, then it maps at least $Q+1\ge 3$ elements of $\F_{Q^2}$ into $\F_{Q^2}$;
since a degree-one rational function can be defined over any field which contains three points and their images, 
it follows that $\ell(x)$ is in $\F_{Q^2}(x)$.
Thus, we may assume that $\ell(x)=(\alpha x-\delta)/(\gamma x-\beta)$ with $\alpha,\beta,\gamma,\delta\in\F_{Q^2}$ and $\alpha\beta\ne \gamma\delta$.
If $\ell(x)$ permutes $\mu_{Q+1}$ then in particular $\mu_{Q+1}$ does not contain $\ell^{-1}(\infty)=\beta/\gamma$.
Henceforth assume that $\mu_{Q+1}$ does not contain $\beta/\gamma$.  Then $\ell(x)$ induces a bijection on $\mu_{Q+1}$
if and only if the numerator of $\ell(x)^{Q+1}-1$ is divisible by $x^{Q+1}-1$.  This numerator is
$(\alpha x-\delta)^{Q+1}-(\gamma x-\beta)^{Q+1}$, which equals
\[
(\alpha^{Q+1}-\gamma^{Q+1})x^{Q+1}-(\alpha^Q\delta-\gamma^Q\beta)x^Q-
(\alpha\delta^Q-\gamma\beta^Q)x+\delta^{Q+1}-\beta^{Q+1}.
\]
This polynomial is divisible by $x^{Q+1}-1$ if and only if it equals $(\alpha^{Q+1}-\gamma^{Q+1})(x^{Q+1}-1)$,
or equivalently
\[
\alpha^Q\delta =\gamma^Q\beta  \quad\text{ and }\quad \alpha^{Q+1}+\delta^{Q+1}=\gamma^{Q+1}+\beta^{Q+1}.
\]
If $\alpha =0$ then, since $\alpha^Q\delta =\gamma^Q\beta $ but $\alpha \beta \ne \gamma\delta$, we must have $\beta =0$, and then $\delta^{Q+1}=\gamma^{Q+1}$, whence
$f(x)=-\delta /(\gamma x)$ where $(-\delta /\gamma)^{Q+1}=1$.  Now suppose $\alpha \ne 0$, so that upon dividing $\alpha,\delta,\gamma,\beta$ by $\alpha$ we may assume that
$\alpha =1$.  Then $\delta =\gamma^Q\beta$ and $\gamma^{Q+1}+\beta^{Q+1}=\alpha^{Q+1}+\delta^{Q+1}=1+\gamma^{Q+1}\beta^{Q+1}$, so that $(\gamma^{Q+1}-1)(\beta^{Q+1}-1)=0$
and thus either $\gamma^{Q+1}=1$ or $\beta^{Q+1}=1$.  But $0\ne \alpha \beta -\gamma\delta =\beta (1-\gamma^{Q+1})$ implies that $\gamma^{Q+1}\ne 1$, so $\beta^{Q+1}=1$.
Finally, in this case the condition $\beta /\gamma \notin\mu_{Q+1}$ asserts that $\gamma \notin\mu_{Q+1}$.
\end{proof}

The next lemma was first proved in \cite{TZ}.

\begin{lemma} \label{lem}
Pick $h\in\F_q[x]$ and integers $d,r>0$ such that $d\mid (q-1)$.  Then $f(x):=x^r h(x^{(q-1)/d})$
permutes\/ $\F_q$ if and only if both
\begin{enumerate}
\item $\gcd(r,(q-1)/d)=1$ \,\text{ and}
\item $x^r h(x)^{(q-1)/d}$ permutes $\mu_d$.
\end{enumerate}
\end{lemma}

This lemma has been used in several investigations of permutation polynomials, for instance see
\cite{MZ,TZ,Z1,Z2,Z3}.
Since the proof of Lemma~\ref{lem} is short, we include it here for the reader's convenience.

\begin{proof}
Write $s:=(q-1)/d$.  For $\zeta\in\mu_s$, we have $f(\zeta x)=\zeta^r f(x)$.  Thus, if $f$
permutes $\F_q$ then $\gcd(r,s)=1$.  Conversely, if $\gcd(r,s)=1$ then the values of $f$ on
$\F_q$ consist of all the $s$-th roots of the values of
\[
f(x)^s=x^{rs}h(x^s)^s.
\]
But the values of $f(x)^s$ on $\F_q$ consist of $f(0)^s=0$ and the values of $g(x):=x^r h(x)^s$
on $(\F_q^*)^s$.  Thus, $f$ permutes $\F_q$ if and only if $g$ permutes $(\F_q^*)^s=\mu_d$.
\end{proof}

We now prove Theorem~\ref{thmb}.

\begin{proof}[Proof of Theorem~\ref{thmb}]
Write $h(x):=({\gamma}x-{\beta})^n-{\gamma}(x-{\gamma}^Q{\beta})^n$ and $r:=n+k(Q+1)$.  By Lemma~\ref{lem}, $f(x)=x^r h(x^{Q-1})$ permutes $\F_{Q^2}$ if and only if
$\gcd(n+k(Q+1),Q-1)=1$ and $g(x):=x^r h(x)^{Q-1}$ permutes $\mu_{Q+1}$.  Henceforth we assume
that $\gcd(n+k(Q+1),Q-1)=1$, or equivalently $\gcd(n+2k,Q-1)=1$; note that this implies $n$ is odd if $Q$ is odd, so that $(-1)^{n}=-1$ in $\F_Q$.

We begin by showing that $h(x)$ has no roots in $\mu_{Q+1}$.  For $\alpha\in\mu_{Q+1}$, if
$h(\alpha)=0$ then $\delta:=({\gamma}\alpha-{\beta})/(\alpha-{\gamma}^Q{\beta})$ satisfies $\delta^n={\gamma}$, so in particular
$\delta\notin\mu_{Q+1}$.  But we compute
\[
\delta^Q = \frac{{\gamma}^Q\alpha^{-1}-{\beta}^{-1}}{\alpha^{-1}-{\gamma}{\beta}^{-1}}
 = \frac{{\gamma}^Q{\beta}-\alpha}{{\beta}-{\gamma}\alpha} = \delta^{-1},
\]
which is impossible since $\delta\notin\mu_{Q+1}$.  Hence $h(x)$ has no roots in $\mu_{Q+1}$, so
$h(\mu_{Q+1})\subseteq\F_{Q^2}^*$, whence $g(\mu_{Q+1})\subseteq\mu_{Q+1}$.  Thus, $g$ permutes $\mu_{Q+1}$ if and only if $g$ is injective on $\mu_{Q+1}$.

Next, for $\alpha\in\mu_{Q+1}$ we compute
\[
h(\alpha)^Q = ({\gamma}^Q\alpha^{-1}-{\beta}^{-1})^n - {\gamma}^Q(\alpha^{-1}-{\gamma}{\beta}^{-1})^n
 = \frac{({\gamma}^Q{\beta}-\alpha)^n - {\gamma}^Q({\beta}-{\gamma}\alpha)^n}{({\beta}\alpha)^n},
\]
so that
\[
g(\alpha) = \alpha^{r-n} {\beta}^{-n} \frac{({\gamma}^Q{\beta}-\alpha)^n - {\gamma}^Q({\beta}-{\gamma}\alpha)^n}{({\gamma}\alpha-{\beta})^n-{\gamma}(\alpha-{\gamma}^Q{\beta})^n}.
\]
Since $r-n=k(Q+1)$ and $\alpha^{Q+1}=1$, it follows that $g$ is injective on $\mu_{Q+1}$ if and only if
\[
G(x):={\beta}\frac{({\gamma}^Q{\beta}-x)^n - {\gamma}^Q({\beta}-{\gamma}x)^n}{({\gamma}x-{\beta})^n-{\gamma}(x-{\gamma}^Q{\beta})^n}
\]
is injective on $\mu_{Q+1}$.  For $\ell(x):=(x-{\gamma}^Q{\beta})/({\gamma}x-{\beta})$, we have
\[
G = \ell^{-1}\circ x^n\circ\ell,
\]
so $G$ is injective on $\mu_{Q+1}$ if and only if $x^n$ is injective on $\ell(\mu_{Q+1})$.
By Lemma~\ref{q+1toq+1} we have $\ell(\mu_{Q+1})=\mu_{Q+1}$, so $x^n$ is injective on this
set if and only if $\gcd(n,Q+1)=1$.  This concludes the proof.
\end{proof}

%##################################################################
%##################################################################

\section{Proof of Theorem~\ref{thma}}

In this section we prove Theorem~\ref{thma}.  As in the previous section, $Q$ is a prime power, and
$\mu_d$ denotes the set of $d$-th roots of unity in $\bar\F_Q$.  We begin by determining the
bijections $\mu_{Q+1}\to\F_Q\cup\{\infty\}$ which are induced by degree-one rational functions.

\begin{lemma} \label{q+1toq} Let $Q$ be a prime power, and let $\ell\in\bar\F_Q(x)$ be a degree-one rational function.
Then $\ell(x)$ induces a bijection from $\mu_{Q+1}$ to $\F_Q\cup\{\infty\}$ if and only if $\ell(x)=({\delta}x-{\beta}{\delta}^Q)/(x-{\beta})$
with ${\beta}\in\mu_{Q+1}$ and ${\delta}\in\F_{Q^2}\setminus\F_Q$.
\end{lemma}

\begin{proof}
If $\ell$ maps $\mu_{Q+1}$ into $\F_Q\cup\{\infty\}$, then $\ell$ maps at least $Q+1\ge 3$ elements of $\F_{Q^2}$
into $\F_{Q^2}\cup\{\infty\}$, so $\ell\in\F_{Q^2}(x)$.  Thus we may write $\ell=({\delta}x-{\gamma})/({\alpha}x-{\beta})$ with ${\delta},{\gamma},{\alpha},{\beta}\in\F_{Q^2}$
and ${\delta}{\beta}\ne {\gamma}{\alpha}$.  Moreover, we may assume that $\ell^{-1}(\infty)={\beta}/{\alpha}$ is in $\mu_{Q+1}$, so that after suitably scaling ${\delta},{\gamma},{\alpha},{\beta}$ we
may assume that ${\alpha}=1$ and ${\beta}\in\mu_{Q+1}$.  Under these hypotheses, $\ell$ induces a bijection from $\mu_{Q+1}$ to
$\F_Q\cup\{\infty\}$ if and only if the numerator of $\ell(x)^Q-\ell(x)$ is divisible by $(x^{Q+1}-1)/(x-{\beta})$.  This numerator is
\begin{align*}
({\delta}^Qx^Q&-{\gamma}^Q)(x-{\beta}) - (x^Q-{\beta}^Q)({\delta}x-{\gamma}) \\
&= ({\delta}^Q-{\delta})x^{Q+1}+({\gamma}-{\beta}{\delta}^Q)x^Q+({\delta}{\beta}^Q-{\gamma}^Q)x + ({\gamma}^Q{\beta}-{\gamma}{\beta}^Q),
\end{align*}
which is congruent mod $x^{Q+1}-1$ to
\[
g(x):=({\gamma}-{\beta}{\delta}^Q)x^Q + ({\delta}{\beta}^Q-{\gamma}^Q)x + ({\delta}^Q-{\delta}+{\gamma}^Q{\beta}-{\gamma}{\beta}^Q).
\]
Since
\[
\frac{x^{Q+1}-1}{x-{\beta}} = \frac{x^{Q+1}-{\beta}^{Q+1}}{x-{\beta}} = \sum_{i=0}^Q x^i {\beta}^{Q-i},
\]
it follows that if $Q>2$ then $g(x)$ is divisible by $(x^{Q+1}-1)/(x-{\beta})$ if and only if $g(x)$ is the zero polynomial, or
equivalently
\[
{\gamma}={\beta}{\delta}^Q \quad\text{ and }\quad {\delta}+{\gamma}{\beta}^Q\in\F_Q.
\]
Since ${\beta}\in\mu_{Q+1}$, the second condition follows from the first, as ${\delta}+{\gamma}{\beta}^Q={\delta}+{\delta}^Q{\beta}^{Q+1}={\delta}+{\delta}^Q$ is in $\F_Q$.
If these conditions hold then ${\delta}{\beta}-{\gamma}={\beta}({\delta}-{\delta}^Q)$ is nonzero precisely when ${\delta}\notin\F_Q$.
Finally, one can easily check that the same conclusion holds when $Q=2$.
\end{proof}

We now prove Theorem~\ref{thma}.

\begin{proof}[Proof of Theorem~\ref{thma}]
Write $h(x):=({\delta}x-{\beta}{\delta}^Q)^n-{\delta}(x-{\beta})^n$ and $r:=n+k(Q+1)$.  By Lemma~\ref{lem}, $f(x)=x^r h(x^{Q-1})$ permutes $\F_{Q^2}$ if and only if
$\gcd(n+k(Q+1),Q-1)=1$ and $g(x):=x^r h(x)^{Q-1}$ permutes $\mu_{Q+1}$.  Henceforth we assume
that $\gcd(n+k(Q+1),Q-1)=1$, or equivalently $\gcd(n+2k,Q-1)=1$; note that this implies $n$ is odd if $Q$ is odd, so that $(-1)^{n}=-1$ in $\F_Q$.

We begin by showing that $h(x)$ has no roots in $\mu_{Q+1}$.  Our hypothesis ${\delta}\notin\F_Q$ implies that
$h({\beta})=(\delta\beta-\beta\delta^Q)^n={\beta}^n({\delta}-{\delta}^Q)^n\ne 0$.
For $\alpha\in\mu_{Q+1}\setminus\{{\beta}\}$, if $h(\alpha)=0$ then $\theta:=({\delta}\alpha-{\beta}{\delta}^Q)/(\alpha-{\beta})$
satisfies $\theta^n={\delta}$, so in particular $\theta\notin\F_Q$.  But we compute
\[
\theta^Q = \frac{{\delta}^Q\alpha^{-1}-{\beta}^{-1}{\delta}}{\alpha^{-1}-{\beta}^{-1}} = \frac{{\delta}^Q{\beta}-\alpha {\delta}}{{\beta}-\alpha} = \theta,
\]
which is a contradiction.  Hence $h(x)$ has no roots in $\mu_{Q+1}$, so $h(\mu_{Q+1})\subseteq\F_{Q^2}^*$,
whence $g(\mu_{Q+1})\subseteq\mu_{Q+1}$.  Thus, $g$ permutes $\mu_{Q+1}$ if and only if $g$ is injective on $\mu_{Q+1}$.

Next, for $\alpha\in\mu_{Q+1}$ we compute
\[
h(\alpha)^Q = ({\delta}^Q\alpha^{-1}-{\beta}^{-1}{\delta})^n - {\delta}^Q(\alpha^{-1}-{\beta}^{-1})^n
 = \frac{({\delta}^Q{\beta}-\alpha {\delta})^n - {\delta}^Q({\beta}-\alpha)^n}{(\alpha{\beta})^n},
\]
so that
\[
g(\alpha) = \alpha^{r-n} {\beta}^{-n} \frac{({\delta}^Q{\beta}-\alpha {\delta})^n - {\delta}^Q({\beta}-\alpha)^n}{({\delta}\alpha-{\beta}{\delta}^Q)^n-{\delta}(\alpha-{\beta})^n}.
\]
Since $r-n=k(Q+1)$ and $\alpha^{Q+1}=1$, it follows that $g$ is injective on $\mu_{Q+1}$ if and only if
\[
G(x):=-{\beta}\frac{({\delta}^Q{\beta}-x {\delta})^n - {\delta}^Q({\beta}-x)^n}{({\delta}x-{\beta}{\delta}^Q)^n-{\delta}(x-{\beta})^n}
\]
is injective on $\mu_{Q+1}$.  For $\ell(x):=({\delta}x-{\beta}{\delta}^Q)/(x-{\beta})$, we have
\[
G = \ell^{-1}\circ x^n\circ\ell.
\]
so $G$ is injective on $\mu_{Q+1}$ if and only if $x^n$ is injective on $\ell(\mu_{Q+1})$.
By Lemma~\ref{q+1toq} we have $\ell(\mu_{Q+1})=\F_Q\cup\{\infty\}$, so $x^n$ is injective on this
set if and only if $\gcd(n,Q-1)=1$.  This concludes the proof.
\end{proof}

%##################################################################
%##################################################################

\section{Proofs of Corollary~\ref{main} and Corollary~\ref{maincor}}

In this section we prove Corollary~\ref{main} and Corollary~\ref{maincor}.

\begin{proof}[Proof of Corollary~\ref{main}]
If $Q\equiv 0\pmod{3}$ then $g(1)=0=g(0)$ so $g(x)$ does not permute $\F_{Q^2}$.
If $Q\equiv 1\pmod{3}$ then put $n=3$ and $\beta=1$, and let $\gamma$ be a primitive cube root of unity in $\F_Q$.
In this case, Theorem~\ref{thmb} says that $(\gamma-1)g(x)$ permutes $\F_{Q^2}$ if and only if $\gcd(3+2k,Q-1)=1$.
Finally, if $Q\equiv 2\pmod{3}$ then put $n=3$ and $\beta=\delta$, where $\delta$ is a primitive cube root of unity in
$\F_{Q^2}$.  In this case, Theorem~\ref{thma} says that $(\delta-1)g(x)$ permutes $\F_{Q^2}$ if and only if
$\gcd(3+2k,Q-1)=1$.
\end{proof}

\begin{proof}[Proof of Corollary~\ref{maincor}]
Items (2) and (3)  follow at once from the cases $k=1$ and $k=0$ of Corollary~\ref{main}.
The case $k=Q-3$ of Corollary~\ref{main} asserts that
$g(x):=x^{Q^2-2Q}+3x^{Q^2-Q-1}-x^{Q^2+Q-3}$ is a permutation polynomial over $\F_{Q^2}$ if and
only if $\gcd(2Q-3,Q-1)=1$, which always holds.  Thus $g(x^{Q^2-2})$ is a permutation polynomial over
$\F_{Q^2}$, as is the reduction of $g(x^{Q^2-2})$ mod $x^{Q^2}-x$.  Since this reduction equals
$x^{2Q-1}+3x^Q-x^{Q^2-Q+1}$, item (1) of Corollary~\ref{maincor} follows.
\end{proof}

%##################################################################
%##################################################################

\section{The main result of \cite{TZHL}}

In this section we give a simple proof of a generalization of \cite[Thm.~1]{TZHL}.  Our
proof is completely different from the one in \cite{TZHL}.  Once again, $\mu_{Q+1}$ denotes
the set of $(Q+1)$-th roots of unity in $\bar\F_Q$.

\begin{thm} \label{genthm}
Let $Q$ be a prime power, let $r$ be a positive integer, and let $\beta$ be a $(Q+1)$-th root of unity
in\/ $\F_{Q^2}$.  Let $h(x)\in\F_{Q^2}[x]$ be a polynomial of degree $d$ such that $h(0)\ne 0$ and
\[
\left(x^d \cdot h(1/x)\right)^Q=\beta \cdot h(x^Q).
\]
Then $f(x):=x^r h(x^{Q-1})$ permutes\/ $\F_{Q^2}$ if and only if all of the following hold:
\begin{enumerate}
\item $\gcd(r,Q-1)=1$
\item $\gcd(r-d,Q+1)=1$
\item $h(x)$ has no roots in $\mu_{Q+1}$.
\end{enumerate}
\end{thm}

\begin{remark} The polynomials $h(x)$ satisfying the hypotheses of Theorem~\ref{genthm} can be
described explicitly in terms of their coefficients.  They are
$h(x)=\sum_{i=0}^d a_i x^i$ where $a_0\ne 0$ and, for $0\le i\le d/2$, we have $a_i\in\F_{Q^2}$
and $a_{d-i}=(\beta a_i)^Q$.
% check: x^d h(1/x) = sum_{i=0}^d a_{d-i} x^i, so its Q-th power is  sum a_{d-i}^Q x^{Qi}.
% If i\le d/2 then  a_{d-i}^Q=\beta a_i.  If d/2<i\le d then 0\le d-i < d/2  so
%  a_i = (\beta a_{d-i})^Q, whence  a_{d-i}^Q = a_i / \beta^Q = \beta a_i.
% Thus  (x^d h(1/x))^Q = \beta h(x^Q), as desired.
\end{remark}

\begin{proof}[Proof of Theorem~\ref{genthm}]
By Lemma~\ref{lem}, we see that $f(x)$ permutes $\F_{Q^2}$ if and only if $\gcd(r,Q-1)=1$ and
$g(x):=x^r h(x)^{Q-1}$ permutes $\mu_{Q+1}$.  We may assume that $h(x)$ has no roots in $\mu_{Q+1}$, since
otherwise $g$ cannot permute $\mu_{Q+1}$.
Then any $\alpha\in\mu_{Q+1}$ satisfies
\[
g(\alpha)=\alpha^r \frac{h(\alpha)^Q}{h(\alpha)} = \alpha^r \frac{h(\alpha^{-Q})^Q}{h(\alpha)}
= \alpha^{r-d}\beta,
\]
so $g$ permutes $\mu_{Q+1}$ if and only if $\gcd(r-d,Q+1)=1$.
\end{proof}

We now illustrate Theorem~\ref{genthm} in the special case $h(x)=x^d+\beta^{-1}$.

\begin{cor} \label{gencor} Let $Q$ be a prime power, let $r$ and $d$ be positive integers, and let $\beta$ be a
$(Q+1)$-th root of unity in\/ $\F_{Q^2}$.  Then $x^{r+d(Q-1)}+\beta^{-1} x^r$ permutes\/ $\F_{Q^2}$
if and only if all of the following hold:
\begin{enumerate}
\item $\gcd(r,Q-1)=1$
\item $\gcd(r-d,Q+1)=1$
\item $(-\beta)^{(Q+1)/\gcd(Q+1,d)}\ne 1$.
\end{enumerate}
\end{cor}

\begin{proof} Since $h(x):=x^d+\beta^{-1}$ satisfies the hypotheses of Theorem~\ref{genthm},
% check:  x^d h(1/x) = 1 + beta^{-1} x^d,  raise to the Q-th power to get  1 + beta x^{Qd},
% which equals  beta  times  h(x^Q).
the Corollary will follow from Theorem~\ref{genthm} once we show that the final conclusion in the Corollary
is equivalent to the final conclusion in the Theorem.  For this, note that $h(x)$ has roots in $\mu_{Q+1}$
if and only if $-\beta^{-1}$ is  in $(\mu_{Q+1})^d$, which equals $\mu_{(Q+1)/\gcd(Q+1,d)}$.
\end{proof}

In case $Q$ is even, Corollary~\ref{gencor} is a refinement of \cite[Thm.~1]{TZHL}.

\begin{comment}
Check that this really does generalize the cited result.  That result is for
x^{s(Q-1)+e} + beta x^{(s-l)(Q-1)+e}  over  F_{Q^2}  in characteristic 2,
where  gcd(s(Q-1)+e,Q^2-1)=1 and gcd(l,Q+1)>1 and gcd(e+l-2s,Q+1)=1
and  beta  is a (Q+1)-th root of unity which is not a gcd(l,Q+1)-th power in the (Q+1)-th roots of unity.
The last condition is the same as my last condition.
My other conditions translate to:
gcd(e,Q-1)=1  and  gcd(e-2s+l, Q+1)=1.
So I replace their  gcd(s(Q-1)+e,Q^2-1)=1 with gcd(s(Q-1)+e,Q-1)=1.
That is, in their notation, I replace gcd(d_1,2^n-1)=1 with gcd(d_1,2^m-1)=1.
\end{comment}

%###############################################################################
%###############################################################################
%###############################################################################

\end{document}